\documentclass[reqno, 12pt]{amsart}
\pdfoutput=1
\makeatletter
\let\origsection=\section \def\section{\@ifstar{\origsection*}{\mysection}} 
\def\mysection{\@startsection{section}{1}\z@{.7\linespacing\@plus\linespacing}{.5\linespacing}{\normalfont\scshape\centering\S}}
\makeatother        
\usepackage{amsmath,amssymb,amsthm}    
\usepackage{mathabx}\changenotsign    
\usepackage{mathrsfs}
\usepackage{bbm, accents} 

\usepackage{xcolor}  	
\usepackage[backref]{hyperref}
\hypersetup{
	colorlinks,
    linkcolor={red!60!black},
    citecolor={green!60!black},
    urlcolor={blue!60!black},
}

\usepackage{bookmark}

\usepackage[abbrev,msc-links,backrefs]{amsrefs} 
\usepackage{doi}

\renewcommand{\PrintDOI}[1]{\doi{#1}}

\usepackage[T1]{fontenc}
\usepackage{lmodern}

\usepackage[english]{babel}
\usepackage[utf8]{inputenc}

\numberwithin{equation}{section}

\theoremstyle{plain}
\newtheorem{thm}{Theorem}[section]
\newtheorem{fact}[thm]{Fact}
\newtheorem{prop}[thm]{Proposition}
\newtheorem{lemma}[thm]{Lemma}

\theoremstyle{definition}
\newtheorem{dfn}[thm]{Definition}
\usepackage{geometry}
\let\theta=\vartheta
\let\rho=\varrho
\let\phi=\varphi
\let\into=\longrightarrow

\def\NN{\mathbbm N}

\def\QQ{\mathbbm Q}

\let\vn=\varnothing
\def\st{\,|\,}
\def\bl{\bigl(}
\def\br{\bigr)}

\newcommand{\seq}[1]{\accentset{\rightharpoonup}{#1}}

\usepackage{enumitem} 
\def\rmlabel{\upshape({\itshape \roman*\,})}

\def\alabel{\upshape({\itshape \alph*\,})}

\let\setminus=\smallsetminus

\begin{document}
\title[The chromatic number of finite type-graphs]{The chromatic number of finite type-graphs}

\author[Christian Avart]{Christian Avart}
\address{Department of Mathematics and Statistics, 
Georgia State University, Atlanta, GA 30303, USA}
\email{cavart@gsu.edu}

\author[Bill Kay]{Bill Kay}
\address{Department of Mathematics and Computer Science, 
Emory University, Atlanta, GA 30322, USA}
\email{bill.w.kay@gmail.com}
\thanks{The second author was supported by NCN grant 2012/06/A/STI/00261 
and by NSF grant {DMS 1301698}}

\author[Christian Reiher]{Christian Reiher}
\address{Fachbereich Mathematik, Universit\"at Hamburg,
  Bundesstra\ss{}e~55, D-20146 Hamburg, Germany}
\email{Christian.Reiher@uni-hamburg.de}

\author[Vojt\v{e}ch R\"{o}dl]{Vojt\v{e}ch R\"{o}dl}
\address{Department of Mathematics and Computer Science, 
Emory University, Atlanta, GA 30322, USA}
\email{rodl@mathcs.emory.edu}
\thanks{The fourth author was supported by NSF grants DMS 1301698 and 1102086.}

\keywords{shift graphs, type-graphs, chromatic number, order types of pairs}
\subjclass[2010]{Primary 05C15, Secondary 05C55}

\begin{abstract}
By a {\it finite type-graph} we mean a graph whose set of vertices is the 
set of all $k$-subsets of $[n]=\{1,2,\ldots, n\}$ for some integers $n\ge k\ge 1$, 
and in which two such sets are adjacent if and only if they realise a certain
order type specified in advance. 
Examples of such graphs have been investigated in a great variety of contexts 
in the literature with particular attention being paid to their chromatic number. 
In recent joint work with Tomasz \L uczak, two of the authors embarked on a systematic
study of the chromatic numbers of such type-graphs, formulated
a general conjecture determining this number up to a multiplicative 
factor, and proved various results of this kind. In this article we fully 
prove this conjecture.
\end{abstract} 

\maketitle

\section{Motivation} \label{sec:intro}
Our goal in this article is to analyse the asymptotic behaviour of the chromatic number 
of certain finite graphs, that are called type-graphs in the sequel. In general the vertex 
set of such a graph is, for some positive integers $n\ge k$, the collection of all $k$-element 
subsets of the set $[n]=\{1,2,\ldots, n\}$. 
Whether two such subsets are to be connected by an edge or not is 
decided solely in terms of the mutual position of their elements or, equivalently,  
it only depends on the order type that this pair of sets realises. Before defining these
type-graphs accurately, we would like to fix some notation concerning order types of
pairs of ordered sets. 
In particular, we shall encode such order types as finite sequences consisting of 
ones, twos, and threes. 
At first sight, allowing rational numbers in the definition that follows might look 
unnecessarily general, but it will turn out to be useful at a later occasion.

\begin{dfn}
Let $X$ and $Y$ be two finite sets of rational numbers with $|X\cup Y|=\ell$
and $X\cup Y =\{z_1,z_2,\ldots,z_\ell\}$, these elements being listed in increasing order.
We say that the {\it order type} of the pair $(X, Y)$ is the sequence 
$\tau=(\tau_1, \ldots, \tau_\ell)$ and set $\tau(X,Y)=\tau$ if  
for every~$i\in[\ell]$ we have
\[
\tau_i=
\begin{cases}
1 &\textnormal{ if } z_i\in X\setminus Y\,, \cr
2 &\textnormal{ if } z_i\in Y\setminus X\,,\cr
3 &\textnormal{ if } z_i\in X\cap Y\,.\cr
\end{cases}
\]
\end{dfn}

For example, given $X = \{1,2,3,5\}$ and $Y = \{3,4,5\}$ we get $\tau(X,Y) = 11323$.
Clearly for any finite sequence $\tau$ consisting of ones, twos, and threes there are 
two finite subsets~$X$ and~$Y$ of $\QQ$ with $\tau=\tau(X, Y)$ and in fact one may 
even find such sets with $X, Y\subseteq\NN$. 

The case most relevant for the definition of type-graphs below is $|X|=|Y|$.
 
\begin{dfn}
Consider two nonnegative integers $k$ and $\ell$. By a {\it type} of {\it width}~$k$
and {\it length}~$\ell$ we mean the order type of a pair $(X, Y)$ with $X, Y\subseteq \QQ$,
$|X|=|Y|=k$, and $|X\cup Y|=\ell$. 
\end{dfn}

So $\tau=123312$ is a type of width $4$ and length $6$ that is realised, e.g., by
$X = \{1,3,4,7\}$ and $Y = \{2,3,4,9\}$. It is not hard to observe that in any
type of width $k$ and length $\ell$ there appear $\ell - k$ ones, $\ell -k$ twos, 
and $2k-\ell$ threes. As a degenerate case we regard the empty sequence $\vn$ as an
{\it empty type} of width and length $0$. A type is said to be {\it trivial} if it
consists of threes only, or in other words if its width equals its length. 

Now we are prepared to define the main objects under consideration in this article.
 
\begin{dfn}
\label{dfn:typegraph}
For a nontrivial type $\tau$ of width $k$ and an integer $n\ge k$, 
the {\it type-graph} $G(n,\tau)$ is the graph with vertex 
set $\binom{[n]}{k}$ in which two vertices $X$ and $Y$ are declared to be adjacent 
if and only if we have $\tau(X,Y)=\tau$ or $\tau(Y, X)=\tau$.
\end{dfn}

Such graphs and their chromatic numbers have been studied in numerous articles. 
For example, it is known that the {\it shift graph} $G(n, 132)$ has chromatic number 
$\lceil \log(n)\rceil$, where the base of the logarithm is $2$. It is straightforward 
to check that these shift graphs are triangle-free, and thus they provide explicit examples
of triangle-free graphs with arbitrarily large chromatic number. More generally, 
Erd\H{o}s and Hajnal~\cite{EH66a} considered the type-graph~$G(n, \sigma_{k})$ with 
\begin{equation}\label{eq:sigmak}
\sigma_k=1\underbrace{3\dots 3}_{k-1}2\,,
\end{equation}
and the infinite analogues of this graph that naturally arise when one replaces the 
finite number~$n$ by an arbitrary cardinal number. 
Concerning the chromatic number of the finite graphs $G(n, \sigma_k)$ they obtained
the following result that we will apply later. 
\enlargethispage{1em}

\begin{thm}[Erd\H{o}s and Hajnal] 
\label{thm:EH}
For any integer $k\ge 2$ we have 
\[
\chi\bl G(n, \sigma_k)\br=\bl 1+o(1)\br\cdot \log_{(k-1)}(n)
\]
as $n$ tends to infinity.
\end{thm}

Here for any $t \in \NN$ and any sufficiently large $n\in\NN$, we denote the $t$-fold 
iterated base~2 logarithm of $n$ by $\log_{(t)}(n)$.
Strictly speaking Erd\H{o}s and Hajnal did mainly focus on the case where $n$ is infinite,
see~\cite{EH66a}*{Lemma~2}, but their method of proof applies to finite values of $n$ as 
well. The thus adapted proof may be found with more details in~\cite{DLR95} or~\cite{MSW15}. 
In the latter reference, the alternative language of ordered Ramsey theory is used. 
We note that the infinite case of 
Theorem~\ref{thm:EH} has applications to the computation of infinite Ramsey 
numbers~\cite{EH66a}*{Theorem~1} and refer the reader interested
in further applications of infinite type-graphs to~\cite{Sp57}, \cite{ER60}, \cite{Pr86}, 
and~\cite{KS05}.

Another interesting consequence of Theorem~\ref{thm:EH} is that it provides us with 
explicit examples of graphs having large chromatic number and large odd girth. 
In fact, any odd cycle contained in $G(n, \sigma_k)$ has at least the length $2k+1$. 
This line of thought was substantially continued by Ne{\v{s}}et{\v{r}}il
and R{\"o}dl, who used unions of general type-graphs in some of their early work on 
structural Ramsey theory, see e.g.~\cite{NR76}.

The problem of determining the chromatic number of general finite type-graphs was 
recently approached in joint work of \L uzcak and two of the current authors~\cite{ALR}. 
The last section of that article contains a conjecture, restated as 
Theorem~\ref{thm:main-b} below,  
that predicts this number asymptotically up to a constant multiplicative factor. 
In particular, this conjecture implies that for each nontrivial type 
$\tau$ there exists a nonnegative integer~$\beta$ with~ 
$\chi\bl G(n,\tau)\br= \Theta\bl\log_{(\beta)}(n)\br$ as $n$ tends to infinity. 
When intending to calculate $\beta$ from $\tau$ the first thing one has to do is 
to express $\tau$ as a product of as many other types as possible. 
The next two definitions help us to talk about this process:

\begin{dfn}
Given two finite sequences $\tau=(\tau_1, \ldots,\tau_{\ell})$ and 
$\tau'=(\tau'_1, \ldots,\tau'_{\ell'})$
we write~$\tau\tau'$ for their {\it concatenation} 
$(\tau_1, \ldots, \tau_\ell, \tau'_1, \ldots, \tau'_{\ell'})$.
\end{dfn}

\begin{dfn} 
A nonempty type is said to be {\it irreducible} if it cannot be written as the concatenation of two nonempty types. 
\end{dfn}

It should be clear that each nonempty type $\tau$ can be written in a unique manner as the
concatenation of several irreducible types. In fact, one finds this unique {\it factorisation} 
of~$\tau$ by keeping track of the numbers of ones and twos already encountered while reading~ 
$\tau$ from left to right, and
starting a new factor at every moment where these two numbers are equal. 
As it will turn out, most of our work concerning $\chi\bl G(n,\tau)\br$ addresses the 
irreducible case. 
Once it is solved, the reducible case reduces to that case. 

In the next section, we describe an algorithm which partitions any given irreducible 
type~$\tau$ into so-called blocks. Notice that if $\tau$ is trivial, i.e., a string of threes, 
we must have $\tau=3$ and in this case the number of blocks is going to be $1$. On the other 
hand, any nontrivial irreducible type is going to be partitioned into at least $2$ blocks.

Our main result on irreducible types states:

\begin{thm}
\label{thm:main-a}
If $\tau$ is a nontrivial irreducible type of width $k$ with $b$ blocks, then 
\[
(1+o(1))\log_{(b-2)}\bl\tfrac nk\br\le \chi\bl G(n,\tau)\br\le 
\bl 2^{(b-2)^2}+o(1)\br\log_{(b-2)}(n)\,
\]
and hence
\[
\chi\bl G(n,\tau)\br= \Theta\bl\log_{(b-2)}(n)\br\,.
\]
\end{thm}

More generally we shall obtain the following:

\begin{thm}
\label{thm:main-b}
Let $\tau=\rho_1\rho_2\cdot\ldots\cdot\rho_t$ be the factorisation of an arbitrary nontrivial 
type $\tau$ into irreducible types. 
Suppose that $\rho_i$ has $b_i$ blocks for $i\in [t]$, and set $b^*=\max(b_1, \ldots, b_t)$. 
Then we have
\[
\chi\bl G(n,\tau)\br= \Theta\bl\log_{(b^*-2)}(n)\br\,.
\]
\end{thm} 

The rest of this article is structured as follows: In Section~\ref{sec:block} we describe the
block algorithm and thus clarify the meaning of our main results. Then the next two sections 
are dedicated to the proofs of the lower and upper bounds appearing in
Theorem~\ref{thm:main-a}. Finally, in Section~\ref{sec:prod} we will deduce 
Theorem~\ref{thm:main-b} by means of a product argument. 

\section{The block algorithm}
\label{sec:block}

In this section we describe an algorithm partitioning the terms of any irreducible 
type~$\tau$ into blocks of consecutive terms. We will call this algorithm the {\it block algorithm} and the partition  
it produces will be referred to as the {\it block decomposition} of~$\tau$. 

As said above, if~$\tau$ is trivial we have~$\tau=3$ by irreducibility. 
In this special case we regard~$\tau$ as consisting of one block only, namely~$\tau$ itself. 
If~$\tau\ne 3$, then the first digit of~$\tau$ is either a one or a two, because otherwise we 
could write~$\tau=3\rho$ for some 
type $\rho\ne\vn$, contrary to the irreducibility of~$\tau$. We call~$\tau$ {\it primary} if it 
starts with a one and {\it secondary} if it starts with a two. 

Given a subsequence $B$ of a type~$\tau$ that consists of consecutive terms, 
we write~$\textbf{1}(B)$ for the total number of ones and threes occurring in~$B$, 
and~$\textbf{2}(B)$ for the total number of twos and threes in~$B$. 

\enlargethispage{1em}

Now we are ready to explain how the block algorithm is applied to any primary 
irreducible type~$\tau$.
Processing $\tau$ from left to right we are to perform the following steps:

\begin{enumerate}[label=\rmlabel]
\item\label{it:ba1} 
The first block $B_1$ consists of all the initial ones appearing in $\tau$.
\item\label{it:ba2} 
In general, if the block $B_{i}$ has just been constructed, 
the next block $B_{i+1}$ consists of the next consecutive digits of~$\tau$ 
such that~$\textbf{2}(B_{i+1})=\textbf{1}(B_{i})$ 
and such that subject to this condition the block~$B_{i+1}$ is as long as possible.
\item\label{it:ba3} 
The algorithm stops when all the terms of~$\tau$ have been placed in a block.
\end{enumerate}

E.g., for the type $\tau=1121112121212222$ we get $B_1=11$, $B_2=211121$, 
$B_3=212122$, and finally $B_4=22$. One may use appropriate spacing to make 
the outcome of the block algorithm notationally visible and write, for instance,
\[
\tau=11 \ \  211121 \ \ 212122 \ \ 22\,.
\]
Similarly the type $131122311222$ decomposes into 
\[
1 \ \  311 \ \ 22311 \ \ 222\,
\]
and for the type $\sigma_4=13332$, that we have already encountered in~\eqref{eq:sigmak}, 
the algorithm produces
\[
\sigma_4=1 \ \ 3 \ \ 3 \ \ 3 \ \ 2\,. 
\]
\begin{fact} 
\label{fact:basic}
When applied to a primary irreducible type $\tau$ the block algorithm does
indeed provide a factorisation $\tau = B_1 B_2 \cdot \ldots \cdot B_b$
of $\tau$ into some nonempty blocks $B_1, \ldots, B_b$, where~$b\ge 2$. 
Moreover, we have $\textbf{1}(B_b)=0$. 
\end{fact}

\begin{proof}
Since $\tau$ starts with a one, rule~\ref{it:ba1} gives us a first block $B_1\ne\vn$.
Now let $i$ be the largest integer for which the block algorithm produces in its first 
$i$ steps some nonempty blocks $B_1, \ldots, B_i$. 
This happens by an initial application of~\ref{it:ba1} followed
by~$i-1$ applications of~\ref{it:ba2}. Let $C$ denote the finite sequence satisfying
\begin{equation}\label{eq:CC}
\tau=B_1\cdot \ldots \cdot B_iC\,. 
\end{equation}
We intend to show that either $C=\vn$ so that the algorithm stops, or 
${0<\textbf{1}(B_i)\le \textbf{2}(C)}$, meaning that the algorithm produces a further nonempty 
block $B_{i+1}$. The latter alternative, however, would contradict the maximality of $i$.

Recall that by construction we have $\textbf{2}(B_1)=0$ and 
$\textbf{1}(B_{j})=\textbf{2}(B_{j+1})$ for all $j\in[i-1]$. 
This yields
\begin{equation}\label{eq:23} 
\textbf{1}(B_1\cdot \ldots \cdot B_{i-1})=\textbf{2}(B_1\cdot \ldots \cdot B_i)\,
\end{equation}
and in combination with~\eqref{eq:CC} and $\textbf{1}(\tau)=\textbf{2}(\tau)$
it follows that ${\textbf{1}(B_i)\le \textbf{1}(B_iC)=\textbf{2}(C)}$.
So if $\textbf{1}(B_i)>0$ we could use \ref{it:ba2} once more to obtain the next 
nonempty block~$B_{i+1}$, contrary to the maximality of $i$.

Thus we must have $\textbf{1}(B_i)=0$ and~\eqref{eq:23} entails that $B_1\cdot \ldots \cdot B_i$ 
is a type. By~\eqref{eq:CC} and the irreducibility of $\tau$ it follows that $C=\vn$, meaning 
that the algorithm stops with a final application of rule~\ref{it:ba3}. Now $b=i$, the
moreover-part was obtained at the beginning of this paragraph, and $b\ge 2$ is clear. 
\end{proof}

So far we have only talked about primary types. For dealing with secondary types
we use the following symmetry: If $\tau$ denotes any finite sequence of ones, twos, and threes,
we write~$\tau'$ for the sequence obtained from $\tau$ by replacing all ones by twos and vice
versa. Evidently if~$\tau$ is a secondary irreducible type, then $\tau'$ is a primary 
irreducible type and thus we already know how to find its block decomposition 
$\tau'=B_1 B_2 \cdot \ldots \cdot B_b$. Now we have $\tau=B_1' B_2' \cdot \ldots \cdot B_b'$ 
and we define this to be the block decomposition of $\tau$. 
In particular, $\tau$ and $\tau'$ have the same numbers of blocks.

Notice that if $\tau(X, Y)=\rho$ holds for some finite sets $X, Y\subseteq\QQ$,
then $\tau(Y, X)=\rho'$ follows. In particular, for any type $\tau$ the two
type-graphs $G(n, \tau)$ and $G(n, \tau')$ are the same and thus it suffices to prove
Theorem~\ref{thm:main-a} for primary $\tau$. 

We conclude this section with two statements concerning irreducible types and the block
algorithm that will be employed in Section~\ref{sec:upper}.

\begin{lemma}
\label{lem:irr}
Suppose that $\tau$ is a primary irreducible type of width $k$ and that $X,Y\subseteq \QQ$ 
are two finite sets with $\tau=\tau(X, Y)$. 
Let $X=\{x_1, \ldots, x_k\}$ and $Y=\{y_1, \ldots, y_k\}$, the
elements being listed in increasing order. Then we have
\begin{enumerate}[label=\alabel]
\item\label{it:IL-a} $x_{i}< y_i$ for all $i\in[k]$
\item\label{it:IL-b} and $x_{i+1}\le y_i$ for all $i\in[k-1]$.
\end{enumerate}
\end{lemma} 

\begin{proof}
Let $\tau=(\tau_1, \ldots, \tau_\ell)$, where $\ell$ denotes the length of $\tau$. We contend
that 
\begin{equation}\label{eq:22} 
\text{if  $i\in [k-1]$ and $x_i\le y_i$, then $x_{i+1}\le y_i$}. 
\end{equation}
To show this, let $y_i$ be the $m$-th element in the increasing enumeration of $X\cup Y$.
In view of $1\le i<k$ we have $1\le m<\ell$ and thus $(\tau_1, \ldots, \tau_m)$ cannot be 
a type due to the irreducibility of $\tau$.
This  in turn yields $|X\cap (-\infty, y_i]|\ne |Y\cap (-\infty, y_i]|=i$. 
But assuming $x_i\le y_i$ the number $|X\cap (-\infty, y_i]|$ is at least $i$, so that 
altogether it must be at least $i+1$, which means that $x_{i+1}\le y_i$. 
This proves~\eqref{eq:22}. 

Next we show \ref{it:IL-a} by induction on $i$. The base case $x_1<y_1$ follows from $\tau$
being primary. For the induction step we suppose that $x_i<y_i$ holds for some $i<k$. 
Then~\eqref{eq:22} entails $x_{i+1}\le y_i<y_{i+1}$, which concludes the argument.

Finally~\ref{it:IL-b} is an immediate consequence of~\eqref{eq:22} and~\ref{it:IL-a}.
\end{proof}  

We now come to the only place in the proof of Theorem~\ref{thm:main-a} where the demand 
from the second rule of the block algorithm that the blocks should end with as many ones 
as possible is utilised. The purpose of the following lemma is that, roughly speaking, it 
tells us how the ``blocks'' of two finite sets $X$ and $Y$ realising an irreducible type 
$\tau$ overlap each other. This will be useful in Subsection~\ref{subsec:Gbn} for embedding
$G(n, \tau)$ into an auxiliary graph whose chromatic number is easier to bound from above. 

\goodbreak

\begin{lemma}
\label{lem:block}
Let $\tau=B_1 B_2 \cdot \ldots \cdot B_b$ be the block decomposition of some 
primary irreducible type whose width is $k$ 
and set $s(i)=\textbf{2}(B_1\cdot\ldots\cdot B_{i})$ for all $i\in [b]$. 
Then for any two sets~$X$ and $Y$ satisfying $\tau=\tau(X, Y)$, say
$X=\{x_1, \ldots, x_k\}$ and $Y=\{y_1, \ldots, y_k\}$ with the elements listed in increasing
order, we have $x_{s(i+1)}< y_{s(i)+1}\le x_{s(i+1)+1}$ for all $i\in [b-2]$.
\end{lemma} 
 
\begin{proof}
Let $X\cup Y=\{z_1, \ldots, z_\ell\}$, the elements again being listed in increasing order.
Fix any~$i\in [b-2]$ and set $\beta=\sum_{j=1}^i|B_j|$. By rule~\ref{it:ba2} of the block 
algorithm the block $B_{i+1}$ cannot start with a one and thus we have $z_{\beta+1}\in Y$.
In combination with 
\[
s(i)=\textbf{2}(B_1\cdot\ldots\cdot B_{i})=|Y\cap (-\infty, z_\beta]|
\]
this yields
\begin{equation}\label{eq:24} 
y_{s(i)+1}=z_{\beta+1}\,.
\end{equation}
Similarly we have
\[
s(i+1)=\textbf{2}(B_1\cdot\ldots\cdot B_{i+1})=\textbf{1}(B_1\cdot\ldots\cdot B_{i})
=|X\cap (-\infty, z_\beta]|
\]
and thus $x_{s(i+1)}\le z_\beta$ as well as $z_{\beta+1}\le x_{s(i+1)+1}$. 
The desired conclusion follows from these two estimates and~\eqref{eq:24}. 
\end{proof} 

\section{The lower bound -- uncolourability}
\label{sec:lower}

In this section we shall prove the lower bound from Theorem~\ref{thm:main-a}. So we intend to
show that a certain graph $G(n, \tau)$ cannot be coloured with a certain ``small'' number 
of colours. 
Recall that for any graph $H$ and any natural number $r$, the statement $\chi(H)>r$ means the 
same as saying that there is no graph homomorphism from $H$ to the $r$-clique $K_r$. 
Thus one strategy to prove such an uncolourability statement is to exhibit a homomorphism from 
some auxiliary graph $G$ to~$H$, with $\chi(G)>r$ already being known. 
So in the light of Theorem~\ref{thm:EH} our task reduces to:

\begin{prop}
\label{prop:low}
For every nontrivial irreducible type $\tau$ of width $k$ with $b$ blocks and every 
integer $n\ge b$
there is a graph homomorphism
\[
\phi\colon G(n, \sigma_{b-1}) \longrightarrow G(kn, \tau)\,.
\] 
\end{prop}

For the construction of such a homomorphism, we will make use of the following

\begin{fact}
\label{fact:getR}
If $B$ denotes a finite sequence of ones, twos, and threes, and $Y\subseteq \QQ$
has size~$\textbf{2}(B)$, then there is a set $X\subseteq \QQ$ with $\tau(X, Y)=B$.
\end{fact} 

This can easily be shown by induction on the number of ones appearing in $B$ and we leave the 
details to the reader.

\begin{proof}[Proof of Proposition~\ref{prop:low}]
As said above we may assume that $\tau$ is primary. Let 
\[
\tau=B_1 B_2 \cdot \ldots \cdot B_b
\]
be the block 
decomposition of $\tau$. We commence by defining recursively an auxiliary sequence 
$R_0, R_1, \ldots, R_b$ of finite subsets of $\QQ$ with 
\begin{equation}\label{eq:dfnR}
|R_{i-1}|=\textbf{2}(B_i) 
\quad \text{ for all } i\in[b]\,.
\end{equation}
Since $B_1$ consists exclusively of ones, such a sequence needs to start with $R_0=\vn$. 
Once~$R_{i-1}$ has been defined for some $i\in [b]$, we use
Fact~\ref{fact:getR} to obtain a set $R_i\subseteq \QQ$ satisfying $\tau(R_i, R_{i-1})=B_i$.
Notice that for $i<b$ this yields $|R_i|=\textbf{1}(B_i)=\textbf{2}(B_{i+1})$, so that the 
construction may be continued. We also get $|R_b|=\textbf{1}(B_b)=0$ and hence $R_b=\vn$
from Fact~\ref{fact:basic}.

In view of~\eqref{eq:dfnR} we have
\begin{equation}\label{eq:sumR}
\sum_{i=0}^{b-1}|R_i|=\sum_{i=1}^{b}\textbf{2}(B_i)=\textbf{2}(\tau)=k
\end{equation}
and thus there exist $k$ rational numbers $\alpha_1< \ldots< \alpha_k$ with
\[
\bigcup_{0\le i<b}R_i\subseteq\{\alpha_1, \ldots, \alpha_k\}\,.
\]
Pulling this situation back to $[k]$ we define $R^*_i=\{j\in [k]\,|\,\alpha_j\in R_i\}$ 
for all $i\in[b-1]$ as well as $R_0^*=R_b^*=\vn$. The main properties of these sets are
\begin{equation}\label{eq:Rprop}
R_i^*\subseteq [k] 
\quad\text{ and }\quad
\tau(R^*_i, R^*_{i-1})=B_i
\quad\text{ for all } i\in [b]\,.
\end{equation}
Now we are ready to define the requested map
\[
\phi\colon\binom{[n]}{b-1}\longrightarrow\binom{[kn]}{k}\,.
\]
Given any integers $h_i$ for $i\in [b-1]$ with $1\le h_1<\ldots<h_{b-1}\le n$
we set
\[
\phi\bl\{h_1, \ldots, h_{b-1}\}\br=
\bigcup_{i\in [b-1]}\bigl\{(h_i-1)k+j\,|\,j\in R_i^*\bigr\}\,.
\]
Due to $R^*_i\subseteq [k]$ the right-hand side of this formula is indeed a subset of $[kn]$
and by~\eqref{eq:sumR} its size is $k$. It remains to check that $\phi$ maps edges of 
$G(n, \sigma_{b-1})$ to edges of $G(kn, \tau)$. To this end let any integers $h_i$ for 
$i\in [b]$ with $1\le h_1<\ldots<h_{b}\le n$ be given. Then by~\eqref{eq:Rprop} we have
\begin{align*}
\tau\bl \phi\bl\{h_1, \ldots, h_{b-1}\}\br, \phi\bl\{h_2, \ldots, h_{b}\}\br\br 
&= \tau(R_1^*, R_0^*)\cdot \tau(R_2^*, R_1^*)\cdot\ldots\cdot \tau(R_b^*, R_{b-1}^*) \\
&= B_1B_2\cdot \ldots \cdot B_b=\tau\,,
\end{align*}
as desired.
\end{proof}

\section{The upper bound -- constructing colourings}
\label{sec:upper}

This entire section is dedicated to the proof of the upper bound from Theorem~\ref{thm:main-a}.
The strategy we use is to embed the type-graph $G(n,\tau)$ into 
some other graph $G_{b-1}(n)$ that depends solely on~$b$ and $n$ but  
not on $\tau$ itself. Thereby the task we are to perform gets reduced to the 
problem of colouring these auxiliary graphs with ``few'' colours and it seems that this 
new problem is more susceptible to an inductive treatment than the old one.

\subsection{Embedding type-graphs}
\label{subsec:Gbn}

We begin by defining the auxiliary graphs $G_b(n)$ mentioned above.

\begin{dfn}\label{dfn:Gbn} 
For any positive integers $b$ and $n$ we set
\[
W_b(n)=\{(x_1, \ldots, x_{2b-1})\st 1\le x_1\le x_2\le \ldots\le x_{2b-1}\le n\}\
\]
and 
\[
V_b(n)=\{(x_1, \ldots, x_{2b-1})\in W_b(n)\st x_1< x_3 < \ldots < x_{2b-1}\}\,.
\]
By $G_b(n)$ we mean the graph with vertex set $V_b(n)$ in which an
unordered pair ${e\subseteq V_b(n)}$ is declared to be an edge if we can write
$e=\{\seq{x}, \seq{y}\}$, $\seq{x}=(x_1, \ldots, x_{2b-1})$, and 
${\seq{y}=(y_1, \ldots, y_{2b-1})}$ such that
\begin{enumerate}[label=\rmlabel]
\item\label{it:Gbn1} $x_1<y_1\le x_3<y_3\le\ldots\le x_{2b-1}<y_{2b-1}$ 
\item\label{it:Gbn2} and $x_{j+1}\le y_j$ for $j\in[2b-2]$.
\end{enumerate}
\end{dfn}

It should perhaps be observed that the conditions \ref{it:Gbn1} and \ref{it:Gbn2} 
from this definition do not determine uniquely how the elements of the 
multiset $\{x_1, \ldots, x_{2b-1}\}\cup\{y_1, \ldots, y_{2b-1}\}$ are ordered.
This makes it more plausible, of course, that many type-graphs embed homomorphically 
into $G_b(n)$ and in fact we have

\begin{thm}\label{thm:homo}
For any nontrivial irreducible type $\tau$ with $b\ge 2$ blocks and every positive integer $n$ 
there is a graph homomorphism $\phi\colon G(n, \tau)\into G_{b-1}(n)$.
\end{thm} 

\begin{proof}
As usual we may assume that $\tau$ is a primary type of width $k$, say. 
Let ${\tau=B_1 \cdot \ldots \cdot B_b}$ be its block decomposition and define 
$s(i)=\textbf{2}(B_1\cdot\ldots\cdot B_{i})$ for any $i\in [b]$. Since 
\[
0=s(1)<s(2)<\ldots<s(b)=k\,,
\]
there is a map 
\[
\phi\colon\binom{[n]}{k}\longrightarrow V_{b-1}(n)
\]
given by 
\[
\phi\bl\{x_1, \ldots, x_k\}\br=\bl x_{s(1)+1}, x_{s(2)}, x_{s(2)+1}, \ldots, 
x_{s(b-1)},x_{s(b-1)+1}\br\,,
\]
whenever $1\le x_1<\ldots< x_k\le n$. So roughly speaking $\phi$ remembers where the 
``blocks'' of such a set $\{x_1, \ldots, x_k\}$ start and end and forgets everything else.

It remains to verify that $\phi$ sends edges
of $G(n, \tau)$ to edges of $G_{b-1}(n)$. For this purpose let any two vertices $X$ and $Y$
of $G(n, \tau)$ with $\tau(X, Y)=\tau$ be given and write $X=\{x_1, \ldots, x_k\}$ as well
as $Y=\{y_1, \ldots, y_k\}$, listing the elements in increasing order. We need to show that
$\{\phi(X), \phi(Y)\}$ is an edge of $G_{b-1}(n)$, i.e., that the clauses \ref{it:Gbn1} 
and~\ref{it:Gbn2} from Definition~\ref{dfn:Gbn} are satisfied. 

Now by Lemma~\ref{lem:irr}~\ref{it:IL-a} we have, in particular, $x_{s(i)+1}<y_{s(i)+1}$
for all $i\in [b-1]$ and Lemma~\ref{lem:block} tells us that $y_{s(i)+1}\le x_{s(i+1)+1}$
holds for all $i\in[b-2]$. Both statements together yield condition~\ref{it:Gbn1} from 
Definition~\ref{dfn:Gbn}.

For the verification of~\ref{it:Gbn2} we consider the cases that the index $j$ appearing
there is odd or even separately. To deal with the case where $j$ is odd we need to check
that $x_{s(i+1)}\le y_{s(i)+1}$ holds for all $i\in [b-2]$ and Lemma~\ref{lem:block} informs
us that this is indeed true. For even $j$ we need that $x_{s(i+1)+1}\le y_{s(i+1)}$ holds
for all $i\in [b-2]$ and this was obtained in Lemma~\ref{lem:irr}~\ref{it:IL-b}.
\end{proof}

Now it is clear that in order to complete the proof of Theorem~\ref{thm:main-a} we just need to
establish the following result. 
\begin{thm}\label{thm:upperG}
For every positive integer $b$ we have 
\[
\chi\bl G_b(n)\br\le \bl 2^{(b-1)^2}+o(1)\br\log_{(b-1)}(n)\,.
\]
\end{thm}

Throughout the rest of this section we deal with the proof this theorem. 
We will proceed by induction on $b$, considering the base cases $b=1$ and $b=2$ separately.
They will be established by statement~\eqref{eq:Gb1} and Lemma~\ref{lem:G2n} below.
The main idea for the induction step is to relate the graphs $G_b(2^n)$ and 
$G_{b-1}(n)$ to each other. 
Roughly speaking, we will show that for any $b\ge 3$ the vertex set of the graph $G_b(2^n)$ 
may be split into about~$2^{2b-3}$ pieces, each of which induces a graph that embeds 
homomorphically into~$G_{b-1}(n)$. A~precise assertion along these lines is provided by 
Proposition~\ref{prop:reduce} below. 
For the construction of half of these homomorphisms it will be helpful 
to bear the following symmetry in mind.

\begin{fact}
\label{fact:eta}
For any positive integers $b$ and $n$ the bijection 
$\eta\colon V_b(2^n)\longrightarrow V_b(2^n)$ given by
\[
(x_1, x_2, \ldots, x_{2b-1})\longmapsto 
\bl (2^n+1)-x_{2b-1}, (2^n+1)-x_{2b-2}, \ldots, (2^n+1)-x_{1}\br
\]
is an automorphism of $G_b(2^n)$.
\end{fact}

We leave the easy proof of this assertion to the reader. 

\subsection{Colouring the auxiliary graphs $G_b(n)$} 
Clearly the graph $G_1(n)$ is nothing else than a clique with $n$ vertices. 
Thus we have
\begin{equation}
\label{eq:Gb1}
\chi\bl G_1(n)\br=n \qquad \text{ for every positive integer } n\,.
\end{equation}
The case $b=2$ of Theorem~\ref{thm:upperG} is technically a lot easier than 
the general case and thus we would like to treat it separately.

\begin{lemma}
\label{lem:G2n}
We have $\chi\bl G_2(n)\br\le 2\lceil\log(n)\rceil-1$ for all integers $n\ge 2$.
\end{lemma}

\begin{proof}
Clearly it suffices to show $\chi\bl G_2(2^k)\br\le 2k-1$ for all positive integers $k$ 
and we shall do so by induction on $k$. The base case $k=1$ poses no difficulty 
because the graph~$G_2(2)$ just consists of two isolated vertices. To handle 
the induction step it is enough to show 
\begin{equation}\label{eq:G2n}
\chi\bl G_2(2m)\br\le\chi\bl G_2(m)\br+2
\quad \text{for all } m\ge 2\,.
\end{equation} 
Bearing this goal in mind we partition the vertex set of $G_2(2m)$ into the four classes
\begin{align*}
A& =\bigl\{(x, y, z)\in V_2(2m) \,|\,z\le m\bigr\}\,, \\
B& =\bigl\{(x, y, z)\in V_2(2m) \,|\,y\le m<z\bigr\}\,, \\
C& =\bigl\{(x, y, z)\in V_2(2m) \,|\,x\le m<y\bigr\}\,, \\
\text{ and } \quad D& =\bigl\{(x, y, z)\in V_2(2m) \,|\,m<x\bigr\}\,.
\end{align*}
We also identify subsets of $V_2(2m)$ with the subgraphs of $G_2(2m)$ that they induce.
Evidently~$A$ is the same as $G_2(m)$, the map $(x, y, z)\longmapsto (x+m, y+m, z+ m)$
provides an isomorphism between $A$ and $D$, and there are no edges between $A$ and $D$.
Therefore $A\cup D$ is a disjoint union of two copies of $G_2(m)$ and we have
$\chi(A\cup D)=\chi\bl G_2(m)\br$. Moreover, using condition~\ref{it:Gbn2} from 
Definition~\ref{dfn:Gbn} it is easy to check that the sets $B$ and $C$ are independent.
This concludes the proof of~\eqref{eq:G2n} and, thus, the proof of Lemma~\ref{lem:G2n}.
\end{proof}

Before we proceed to the colouring of $G_b(2^n)$ for $b\ge 3$ we introduce 
some auxiliary functions.

\begin{lemma}
\label{lem:f}
Given any integers $x$ and $y$ with $1\le x<y$ there exist a positive integer $f$ and an 
odd positive integer $q$ such that 
\[
(q-1)\cdot 2^{f-1}<x\le q\cdot 2^{f-1}<y\le (q+1)\cdot 2^{f-1}\,.
\]
Moreover, $f$ and $q$ are uniquely determined by $x$ and $y$ so that we may write 
$f=f(x, y)$ as well as $q=q(x, y)$.
\end{lemma}

\begin{proof}
Let us first prove the existence of $f$ and $q$. To this end, we pick an integer $n$ 
with $y\le 2^n$. Then we expand $x-1$ and $y-1$ in the binary system using $n$ digits 
and allowing leading zeros. Say that this yields $x-1=x_{n-1}\ldots x_1x_0$ and 
$y-1=y_{n-1}\ldots y_1y_0$. Next we compare these expansions from left to right and let
$x_{f-1}\ne y_{f-1}$ be the first place where they differ. 
Notice that $x<y$ entails $x_{f-1}=0$ and $y_{f-1}=1$.
Finally we let $q$ be the number with binary representation $q=x_{n-1}\ldots x_f1$.

So formally we have 
\[
x-1=\sum_{i=0}^{n-1}x_i\cdot 2^i\,, \quad 
y-1=\sum_{i=0}^{n-1}y_i\cdot 2^i\,, \quad 
q=1+\sum_{i=1}^{n-f}x_{f+i-1}2^i
\]
and $x_j=y_j$ for $j\in [f, n-1]$. Clearly, $q$ is odd and 
\[
(q-1)\cdot 2^{f-1}\le x-1< q\cdot 2^{f-1}\le y-1< (q+1)\cdot 2^{f-1}\,,
\]
wherefore $f$ and $q$ are as desired. 

\medskip
\begin{center}
\begin{tabular}{|c||c|c|c|c|c|c|c|c|}
\hline
      & $2^{n-1}$ & $2^{n-2}$ & $\ldots$ & $2^f$ & $2^{f-1}$ & $2^{f-2}$ & $\ldots$  & $1$ \\ \hline \hline
$x-1$ & $x_{n-1}$ & $x_{n-2}$ & $\ldots$ & $x_f$ & $0$       & $x_{f-2}$ & $\ldots$  & $x_0$ \\ \hline
$y-1$ & $x_{n-1}$ & $x_{n-2}$ & $\ldots$ & $x_f$ & $1$       & $y_{f-2}$ & $\ldots$  & $y_0$ \\ \hline
$q\cdot 2^{f-1}$ & $x_{n-1}$ & $x_{n-2}$ & $\ldots$ & $x_f$ & $1$       & $0$ & $\ldots$  & $0$ \\ \hline
\end{tabular}
\end{center}

\medskip

The uniqueness of $f$ and $q$ may likewise be shown by studying the binary expansions 
of~$x-1$ and $y-1$. An alternative argument proceeds as follows: 

Given $x$ and $y$, 
let $(f, q)$ and $(f', q')$ be two pairs with the 
requested properties. Due to symmetry we may suppose $f\le f'$. Now we have 
$(q-1)\cdot 2^{f-1}<x\le q'\cdot 2^{f'-1}$ and consequently $q\le q'\cdot 2^{f'-f}$.
Similarly $q'\cdot 2^{f'-1}<y\le (q+1)2^{f-1}$ yields $q'\cdot 2^{f'-f}\le q$. 
The combination of both estimates reveals $q= q'\cdot 2^{f'-f}$ but, since $q$ is odd, 
this if only possible if $f=f'$ and~$q=q'$.
\end{proof}

We would like to point out that the uniqueness of $f$ and $q$ will be 
essential throughout the following arguments.
By redoing the above proof of this uniqueness more carefully one can show the following 
monotonicity property of the function $f$.

\begin{lemma}
\label{lem:fmono}
For any three positive integers $x$, $y$, and $z$ such that $x<y\le z$ the inequality 
${f(x, y)\le f(x, z)}$ holds.
\end{lemma}

\begin{proof}
For brevity we set $f=f(x, y)$, $q=q(x, y)$, $f'=f(x, z)$, and $q'=q(x, z)$. 
Arguing indirectly we assume $f'<f$. Now $(q'-1)\cdot 2^{f'-1}<x\le q\cdot 2^{f-1}$ 
entails $q'\le q\cdot 2^{f-f'}$ and similarly 
$q\cdot 2^{f-1}< y\le z\le (q'+1)\cdot 2^{f'-1}$ leads to $q\cdot 2^{f-f'}\le q'$. 
Hence we must have $q'=q\cdot 2^{f-f'}$, contrary to the fact that $q'$ is odd.  
\end{proof}

The following will be a standard argument later on.

\begin{lemma}
\label{lem:standard}
For any positive integers $x<y\le z$ with $f(x,y)=f(x,z)$ we have
$q(x,y)=q(x,z)$ and, consequently,
\[
(q-1)\cdot 2^{f-1}<x\le q\cdot 2^{f-1}<y\le z\le (q+1)\cdot 2^{f-1}\,,
\]
where $f=f(x, y)=f(x,z)$ and $q=q(x,y)=q(x,z)$.
\end{lemma}

\begin{proof}
Define $q=q(x,y)$. Lemma~\ref{lem:f} gives 
\[
(q-1)\cdot 2^{f-1}<x\le q\cdot 2^{f-1}<y\le (q+1)\cdot 2^{f-1}
\]
and thus $q\cdot 2^{f-1}$ is the least multiple of $2^{f-1}$ which is at least $x$.
Due to $f=f(x,z)$ this yields $q(x,z)=q$ and hence $z\le (q+1)\cdot 2^{f-1}$.
\end{proof}

Next we record another property of $f$ that shall be utilised later.

\begin{lemma}
\label{lem:strange}
If any four positive integers $t$, $x$, $y$, and $z$ satisfy $t\le x<y\le z$ and 
${f(x, y)=f(x, z)}$, then $f(t, y)=f(t, z)$ holds as well. 
\end{lemma}

\begin{proof}
Setting $f=f(t, z)$ and $q=q(t, z)$ we get 
\[
(q-1)\cdot 2^{f-1}<t\le q\cdot 2^{f-1}<z\le (q+1)\cdot 2^{f-1}
\]
from the definition of these quantities. 

Of course the claim would easily follow from
${q\cdot 2^{f-1}<y}$. So from now on we may assume $y\le q\cdot 2^{f-1}$ towards 
contradiction. This yields
\[
(q-1)\cdot 2^{f-1}<t\le x< y\le q\cdot 2^{f-1}<z\le (q+1)\cdot 2^{f-1}\,,
\]
and, in particular, we obtain $f(x, z)=f$ but $f(x, y)\ne f$, thus reaching a contradiction.  
\end{proof}

To conclude our dicussion of the auxiliary functions $f$ and $q$ we state how they interact
with the map $\eta$ introduced in Fact~\ref{fact:eta}.

\begin{fact}
\label{fact:f-reflect}
For any integers $x$ and $y$ with $1\le x<y\le 2^n$ we have
\begin{align*}
f(x, y)& \in [n]\,,\\
f(2^n+1-y, 2^n+1-x) & =f(x, y)\,, \\
\text{ and } \quad
q(2^n+1-y, 2^n+1-x) &=2^{n+1-f}-q(x, y)\,.
\end{align*}
\end{fact}

Again we leave the straightforward verification to the reader.
We may now return to the problem of colouring the graphs $G_b(2^n)$.

\begin{prop}
\label{prop:reduce}
We have
\[
\chi\bl G_b(2^n)\br\le (2b-6)+2^{2b-3}\,\chi\bl G_{b-1}(n)\br
\]
for any integers $n\ge b\ge 3$.
\end{prop}

\begin{proof} 
For any vertex $\seq{x}=(x_1, x_2, \ldots, x_{2b-1})$ of $G_b(2^n)$ we use the abbreviations  
\begin{align*}
f(\seq{x})   &=f(x_1, x_{2b-1})\,, \\
q(\seq{x})   &=q(x_1, x_{2b-1})\,, \\
T^-(\seq{x}) &=\bl q(\seq{x})-1\br \cdot 2^{f(\,\seq{x}\,)-1}\,, \\
T(\seq{x})   &=q(\seq{x})\cdot 2^{f(\,\seq{x}\,)-1}\,, \\
\text{ and } \quad
T^+(\seq{x}) &=\bl q(\seq{x})+1\br \cdot 2^{f(\,\seq{x}\,)-1}\,.
\end{align*}
Recall that by Lemma~\ref{lem:f} we have 
\begin{equation}\label{eq:TTT}
T^-(\seq{x})<x_1\le T(\seq{x})<x_{2b-1}\le T^+(\seq{x})
\end{equation}
for any such vertex~$\seq{x}$ 
and in the first steps of the current proof we will distinguish these vertices according
to the position of their other entries $x_i$ with respect to $T(\seq{x})$. To begin with, we
partition $V_b(2^n)$ into three sets,
\begin{equation}\label{eq:partGbn}
V_b(2^n)=A\cup B\cup C\,,
\end{equation}
that are defined by 
\begin{align*}
A& =\bigl\{\seq{x}=(x_1, x_2, \ldots, x_{2b-1}) \in V_b(2^n) \,|\,
 x_{2b-3}\le T(\seq{x})  \bigr\}\,, \\
B& =\bigl\{\seq{x}=(x_1, x_2, \ldots, x_{2b-1}) \in V_b(2^n) \,|\,
 x_{3}\le T(\seq{x}) < x_{2b-3}  \bigr\}\,, \\
\text{ and } \quad 
C& =\bigl\{\seq{x}=(x_1, x_2, \ldots, x_{2b-1}) \in V_b(2^n) \,|\,
 T(\seq{x})< x_{3}  \bigr\}\,. \\
\end{align*}
Again we identify subsets of $V_b(2^n)$ with the corresponding induced subgraphs 
of $G_b(2^n)$. We will use different colours for these three sets and commence 
by colouring $B$. This set may be partitioned further into
\[
B=B_3\cup B_4\cup \ldots \cup B_{2b-4}\,,
\]
where
\[
B_i=\bigl\{\seq{x}=(x_1, x_2, \ldots, x_{2b-1}) \in V_b(2^n) \,|\,
 x_{i}\le T(\seq{x}) < x_{i+1}  \bigr\}
\]
for any integer index $i\in [3, 2b-4]$. We claim that each of these $2b-6$ sets is independent.
To show this suppose that $\{\seq{x},\seq{y}\}$ was an edge of $G_b(2^n)$ with 
$\seq{x},\seq{y}\in B_i$ for some $i\in [3, 2b-4]$. Let the notation be as in
Definition~\ref{dfn:Gbn}. By $\seq{x}\in B$, inequality~\ref{it:Gbn1} from 
Definition~\ref{dfn:Gbn}, and by~\eqref{eq:TTT} we have
\[
T^-(\seq{x})<x_1 <  y_1\le x_3 \le T(\seq{x})
<x_{2b-3}  <  y_{2b-3}\le x_{2b-1} \le T^+(\seq{x})\,,
\]
whence $f(y_1, y_{2b-3})=f(\seq{x})$ and $q(y_1, y_{2b-3})=q(\seq{x})$. Due to $\seq{y}\in B$
this yields $f(\seq{y})=f(\seq{x})$ and $q(\seq{y})=q(\seq{x})$. For this reason 
$\seq{x}, \seq{y}\in B_i$ implies $y_i\le T(\seq{y})= T(\seq{x})< x_{i+1}$, 
contrary to part~\ref{it:Gbn2} from Definition~\ref{dfn:Gbn}. So the sets $B_i$ are indeed independent and we obtain
\begin{equation}\label{eq:chiB}
\chi(B)\le 2b-6\,.
\end{equation}
This accounts for the summand $2b-6$ on the right-hand side of our claim and we may proceed 
with analysing $A$ and $C$. Using Fact~\ref{fact:f-reflect} it is not hard to check that
the map $\eta$ from Fact~\ref{fact:eta} constitutes an isomorphism between $A$ and $C$,
wherefore
\begin{equation}\label{eq:chiC}
\chi(A)=\chi(C)\,.
\end{equation}
Now by~\eqref{eq:partGbn},~\eqref{eq:chiB}, and~\eqref{eq:chiC} we have 
\[
\chi\bl G_b(2^n)\br\le \chi(A)+\chi(B)+\chi(C)\le (2b-6)+2\,\chi(A)
\]
and thus to finish the 
current proof we just need to show
\begin{equation}\label{eq:chiA}
\chi(A)\le 2^{2b-4}\,\chi\bl G_{b-1}(n)\br\,.
\end{equation}

The main idea for proving this is to split $A$ into at most~$2^{2b-4}$ further sets, 
each of which is either independent or has the property of being homomorphically mapped 
into $G_{b-1}(n)$ by a certain function $\phi$ that is to be introduced next.
Observe that by the first statement from Fact~\ref{fact:f-reflect} and by Lemma~\ref{lem:fmono}
there is a map $\phi\colon A\longrightarrow W_{b-1}(n)$ defined by
\[
\phi(x_1, x_2, \ldots, x_{2b-1})=\bl f(x_1, x_3), f(x_1, x_4), \ldots, f(x_1, x_{2b-1})\br
\]
for any $(x_1, x_2, \ldots, x_{2b-1})\in A$. 
 
We call two vertices $\seq{x}=(x_1, \ldots, x_{2b-1})$ and 
$\seq{y}=(y_1, \ldots, y_{2b-1})$ from $A$ equivalent and write $\seq{x}\sim \seq{y}$
if for every integer $i\in [3, 2b-2]$ we have 
\[
f(x_1, x_i)=f(x_1, x_{i+1}) \,\, \iff \,\, f(y_1, y_i)=f(y_1, y_{i+1})\,.
\]
It is plain that equivalence is an equivalence relation 
and that the number of its equivalence classes is at most $2^{2b-4}$. 
Thus to conclude the proof of~\eqref{eq:chiA} we just need to verify the following statement:
\begin{equation}\label{eq:homo}
\text{If } \seq{x}, \seq{y}\in A,\, \seq{x}\sim \seq{y},\,
\text{and } \{\seq{x},\seq{y}\}\in E\bl G_b(2^n)\br, \,
\text{then } \{\phi(\seq{x}), \phi(\seq{y})\}\in E\bl G_{b-1}(n)\br\,.
\end{equation}

So let any two equivalent vertices $\seq{x}$ and $\seq{y}$ from $A$ be given and suppose that
they are connected by an edge of $G_b(2^n)$, the notation for this being as in 
Definition~\ref{dfn:Gbn}. For any $i\in [2b-3]$ we set 
\begin{equation}\label{eq:alpha}
\alpha_i=f(x_1, x_{i+2})
\quad \text{and} \quad
\beta_i=f(x_1, y_{i+2})\,.
\end{equation}

Notice that there is no misprint in the last formula -- it is true that 
$\beta_i=f(y_1, y_{i+2})$ holds as well, and actually this fact is very relevant to our 
main concern, but it will only be shown at a rather late moment of our argument. 

Combining the assumption that $\{\seq{x},\seq{y}\}$ be an edge of $G_b(2^n)$ 
with Lemma~\ref{lem:fmono} we infer
\begin{equation}\label{eq:weakodd}
\alpha_1\le \beta_1\le \alpha_3\le \beta_3\le\ldots\le \alpha_{2b-3}\le \beta_{2b-3}
\end{equation}
as well as
\begin{equation}\label{eq:sec}
\alpha_{j+1}\le \beta_j \text{ for } j\in[2b-4]\,.
\end{equation}
 
Next we would like to show
\begin{equation}\label{eq:last}
\alpha_{2b-3}<\beta_{2b-3}\,.
\end{equation}
Assume contrariwise that $\alpha_{2b-3}=\beta_{2b-3}$, i.e., $f(\seq{x})=f(x_1, y_{2b-1})$. 
Lemma~\ref{lem:standard} yields 
\[
	T^-(\seq{x})<x_1\le T(\seq{x})< x_{2b-1}<y_{2b-1}\le T^+(\seq{x})\,,
\]
so in combination with $\{\seq{x},\seq{y}\}$ being an edge and with $\seq{x}\in A$ we obtain
\[
	T^-(\seq{x})<x_1<y_1\le x_{2b-3}\le T(\seq{x})< x_{2b-1}
	\le y_{2b-2}\le y_{2b-1}\le T^+(\seq{x})\,.
\]
It follows that $T(\seq{y})=T(\seq{x})$ and $f(y_1, y_{2b-2})=f(y_1, y_{2b-1})=f(\seq{x})$.
Using $\seq{x}\sim\seq{y}$ we may deduce $f(x_1, x_{2b-2})=f(x_1, x_{2b-1})$.
Now Lemma~\ref{lem:standard} shows that $q(x_1, x_{2b-2})=q(x_1, x_{2b-1})$
holds as well and consequently we have $T(\seq{x})< x_{2b-2}\le y_{2b-3}$. 
Thus we get a contradiction to 
$\seq{y}\in A$, whereby~\eqref{eq:last} is proved.

Extending this result we contend that more generally we have
\begin{equation}\label{eq:oddsharp}
\alpha_{i}<\beta_{i} \quad \text{for all } i\in [2b-3]\,.
\end{equation}
Arguing indirectly again, we let $i$ denote the largest counterexample to this claim.
Notice that~\eqref{eq:last} tells us $i\le 2b-4$. Set $q=q(x_1, x_{i+2})$, 
$T^-=(q-1)\cdot 2^{\alpha_{i}-1}$, $T=q\cdot 2^{\alpha_{i}-1}$, and 
$T^+=(q+1)\cdot 2^{\alpha_{i}-1}$. Due to Lemma~\ref{lem:standard} our indirect assumption 
$\alpha_{i}=\beta_{i}$
entails
\[
T^-<x_1\le T<x_{i+2}\le y_{i+2}\le T^+\,,
\]
which together with $x_{i+2}\le x_{i+3}\le y_{i+2}$ shows 
$f(x_1, x_{i+2})=f(x_1, x_{i+3})$. Now~$\seq{x}\sim\seq{y}$ discloses
$f(y_1, y_{i+2})=f(y_1, y_{i+3})$ and by Lemma~\ref{lem:strange} it follows that
$f(x_1, y_{i+2})=f(x_1, y_{i+3})$. Using Lemma~\ref{lem:standard} again we obtain
\[
T^-<x_1\le T<x_{i+3}\le y_{i+3}\le T^+
\]
and thus $\alpha_{i+1}=\beta_{i+1}$, contrary to the maximality of~$i$. 
Thereby~\eqref{eq:oddsharp} is proved as well.

Now we are ready to confirm the alternative definition of $\beta_i$ announced above.
That is, for any $i\in [2b-3]$ we claim
\begin{equation}\label{eq:beta}
\beta_i=f(x_1, y_{i+2})=f(y_1, y_{i+2}) \,.
\end{equation}
To see this, set $q=q(x_1, y_{i+2})$, 
$S^-=(q-1)\cdot 2^{\beta_i-1}$, $S=q\cdot 2^{\beta_i-1}$, and 
$S^+=(q+1)\cdot 2^{\beta_i-1}$. Now
\[
S^-<x_1\le S<y_{i+2}\le S^+
\]
and $x_3<y_3\le y_{i+2}$. Hence $S<x_3$ would entail 
\[
S^-<x_1\le S<x_3\le S^+
\]
and, consequently, $\alpha_1=f(x_1, x_3)=\beta_i\ge \beta_1$, which contradicts the  
case $i=1$ of~\eqref{eq:oddsharp}. This proves $x_1<y_1\le x_3\le S$,
which in turn establishes~\eqref{eq:beta}.

Putting everything together, the equations~\eqref{eq:alpha} and~\eqref{eq:beta}
yield
\[
\phi(\seq{x})=(\alpha_1, \alpha_2, \ldots, \alpha_{2b-3}) 
\quad \text{ and } \quad
\phi(\seq{y})=(\beta_1, \beta_2, \ldots, \beta_{2b-3}) 
\]
and by~\eqref{eq:oddsharp} we may strengthen~\eqref{eq:weakodd} to
\[
\alpha_1< \beta_1\le \alpha_3< \beta_3\le\ldots\le \alpha_{2b-3}< \beta_{2b-3}\,.
\]
In particular, this shows that $\phi(\seq{x})$ and $\phi(\seq{y})$ are indeed vertices
of $G_{b-1}(n)$ and together with~\eqref{eq:sec} it further shows that these two vertices 
are adjacent. This concludes the proof of~\eqref{eq:homo} and, hence, the proof of 
Proposition~\ref{prop:reduce}.
\end{proof}

To summarise, we would like to emphasise again that~\eqref{eq:Gb1}, Lemma~\ref{lem:G2n}
and Proposition~\ref{prop:reduce} taken together yield an easy proof of 
Theorem~\ref{thm:upperG} by induction on $b$.
Besides, the combination of Proposition~\ref{prop:low}, 
Theorem~\ref{thm:homo}, and Theorem~\ref{thm:upperG} implies 
Theorem~\ref{thm:main-a}.
 
\section{Reducible types}
\label{sec:prod}

Having thus said everything we want to say about the chromatic number of 
irreducible type-graphs, we devote the present section to the proof of 
Theorem~\ref{thm:main-b}. So we consider 
any nontrivial type $\tau$ and let $\tau=\rho_1\rho_2\cdot\ldots\cdot\rho_t$ be 
its factorisation into irreducible types. For each $i\in[t]$ the number of blocks into 
which $\rho_i$ decomposes is denoted by $b_i$ and we set $b^*=\max(b_1, \ldots, b_t)$.
Finally, let $k$ be the width of $\tau$ and let $\rho_i$ have width $k_i$ for $i\in [t]$.

The notation introduced up to this moment will be used throughout this section without 
being repeated in the numbered statements that will occur.

Recall that our goal is to show
\begin{equation*}
\chi\bl G(n,\tau)\br= \Theta\bl\log_{(b^*-2)}(n)\br\,.
\end{equation*}
Here we have $b^*\ge 2$ because otherwise each factor $\rho_i$ of $\tau$ would have to be 
equal to~$3$, meaning that $\tau$ were trivial. Again we treat the lower bound and the 
upper bound separately, but this time the latter is easier, so we start with it. 

\begin{fact}
For every $i\in [t]$ and every integer $n\ge k$ there is a graph homomorphism
\[
\phi_i\colon G(n,\tau) \longrightarrow G(n, \rho_{i})\,.
\]
\end{fact}

\begin{proof} 
Set $r=\textbf{1}(\rho_1\cdot\ldots\cdot\rho_{i-1})$
and $s=\textbf{1}(\rho_1\cdot\ldots\cdot\rho_{i})$. Clearly $\rho_i$ has width $k_i=s-r$,
and, since $\rho_1, \ldots, \rho_i$ are types, we also have 
$r=\textbf{2}(\rho_1\cdot\ldots\cdot\rho_{i-1})$ and 
$s=\textbf{2}(\rho_1\cdot\ldots\cdot\rho_{i})$. Now it easy to confirm that the map
\[
\phi_i\colon \binom{[n]}{k} \longrightarrow \binom{[n]}{k_i}
\]
given by 
\[
\phi\bl\{x_1, \ldots, x_{k}\}\br=
\{x_{r+1}, \ldots, x_s\}
\]
whenever $1\le x_1<x_2<\ldots<x_k\le n$ is as desired.
\end{proof} 

Applying this, in particular, to some index $i^*\in [t]$ with $b_{i^*}=b^*$ we may deduce the
following by means of Theorem~\ref{thm:main-a}.

\begin{fact}
\label{fact:upp}
As $n$ tends to infinity we have
\begin{equation}
\label{eq:52}
\chi\bl G(n,\tau)\br\le \bl 2^{(b^*-2)^2}+o(1)\br\log_{(b^*-2)}(n)\,. 
\end{equation}
\end{fact}

In the other direction, we will use Proposition~\ref{prop:low} to embed the generalised shift
graph $G\bl n, \sigma_{b^*-1}\br$ homomorphically into $G(kn, \tau)$. 

\begin{fact}
\label{fact:low}
For every integer $n\ge b^*$ there is a graph homomorphism
\[
\psi\colon G(n, \sigma_{b^*-1}) \longrightarrow G(kn, \tau)
\]
and, consequently, we have
\begin{equation}
\label{eq:final}
(1+o(1))\log_{(b^*-2)}\bl\tfrac nk\br\le \chi\bl G(n,\tau)\br\,.
\end{equation}
\end{fact}

\begin{proof} 
Let $I=\{i\in [t]\,|\,\rho_i\ne 3\}$ and write $c_i=\sum_{j=1}^ik_j$ for every integer 
$i\in [0, t]$. Recall that we know from Proposition~\ref{prop:low} that for every index 
$i\in I$ there exists a homomorphism 
$\psi_i\colon G(n, \sigma_{b_i-1}) \longrightarrow G(k_in, \rho_i)$.
Utilising these, we define for each $i\in [t]$ a map
\[
\widehat{\psi}_i\colon \binom{[n]}{b^*-1} \longrightarrow \binom{[c_{i-1}n+1, c_in]}{k_i}
\]
by stipulating
\[
\widehat{\psi}_i\bl\{h_1, \ldots, h_{b^*-1}\}\br=
\begin{cases}
c_{i-1}n+\psi_i\bl\{h_1, \ldots, h_{b_i-1}\}\br &\textnormal{ if } i\in I, \cr
\{c_in\} &\textnormal{ if } i\not\in I\,,\cr
\end{cases}
\]
whenever $1\le h_1<\ldots<h_{b-1}\le n$, where the addition of a number to a set in the upper case is to be performed ``elementwise''. We leave it to the reader to check that the map
\[
\psi\colon \binom{[n]}{b^*-1} \longrightarrow \binom{[kn]}{k}
\]
given by 
\[
\psi(X)=\bigcup_{i\in [t]}\widehat{\psi}_i(X)
\]
for all $X\in \binom{[n]}{b^*-1}$ is indeed a homomorphism from $G(n, \sigma_{b^*-1})$ to
$G(kn, \tau)$. 

Formula~\eqref{eq:final} follows from the mere existence of $\psi$ and from 
Theorem~\ref{thm:EH}.
\end{proof}

Owing to ~\eqref{eq:52} and \eqref{eq:final} the proof of Theorem~\ref{thm:main-b} is 
complete.


\begin{bibdiv}
\begin{biblist}

\bib{ALR}{article}{
   author={Avart, Christian},
   author={{\L}uczak, Tomasz},
   author={R{\"o}dl, Vojt{\v{e}}ch},
   title={On generalized shift graphs},
   journal={Fund. Math.},
   volume={226},
   date={2014},
   number={2},
   pages={173--199},
   issn={0016-2736},
   review={\MR{3224120}},
   doi={10.4064/fm226-2-6},
}

\bib{DLR95}{article}{
   author={Duffus, Dwight},
   author={Lefmann, Hanno},
   author={R{\"o}dl, Vojt{\v{e}}ch},
   title={Shift graphs and lower bounds on Ramsey numbers $r_k(l;r)$},
   journal={Discrete Math.},
   volume={137},
   date={1995},
   number={1-3},
   pages={177--187},
   issn={0012-365X},
   review={\MR{1312451 (95j:05136)}},
   doi={10.1016/0012-365X(93)E0139-U},
}

\bib{EH66a}{article}{
   author={Erd{\H{o}}s, P.},
   author={Hajnal, A.},
   title={On chromatic number of infinite graphs},
   conference={
      title={Theory of Graphs},
      address={Proc. Colloq., Tihany},
      date={1966},
   },
   book={
      publisher={Academic Press, New York},
   },
   date={1968},
   pages={83--98},
   review={\MR{0263693 (41 \#8294)}},
}


\bib{ER60}{article}{
   author={Erd{\H{o}}s, P.},
   author={Rado, R.},
   title={A construction of graphs without triangles having preassigned
   order and chromatic number},
   journal={J. London Math. Soc.},
   volume={35},
   date={1960},
   pages={445--448},
   issn={0024-6107},
   review={\MR{0140433 (25 \#3853)}},
}
	
\bib{KS05}{article}{
   author={Komj{\'a}th, P{\'e}ter},
   author={Shelah, Saharon},
   title={Finite subgraphs of uncountably chromatic graphs},
   journal={J. Graph Theory},
   volume={49},
   date={2005},
   number={1},
   pages={28--38},
   issn={0364-9024},
   review={\MR{2130468 (2005k:05096)}},
   doi={10.1002/jgt.20060},
}

\bib{Pr86}{article}{
   author={Preiss, David},
   author={R{\"o}dl, Vojt{\v{e}}ch},
   title={Note on decomposition of spheres in Hilbert spaces},
   journal={J. Combin. Theory Ser. A},
   volume={43},
   date={1986},
   number={1},
   pages={38--44},
   issn={0097-3165},
   review={\MR{859294 (87k:05083)}},
   doi={10.1016/0097-3165(86)90020-8},
}

\bib{MSW15}{article}{
   author={Milans, Kevin G.},
   author={Stolee, Derrick},
   author={West, Douglas B.},
   title={Ordered Ramsey theory and track representations of graphs},
   journal={J. Comb.},
   volume={6},
   date={2015},
   number={4},
   pages={445--456},
   issn={2156-3527},
   review={\MR{3382604}},
   doi={10.4310/JOC.2015.v6.n4.a3},
}
	
\bib{NR76}{article}{
   author={Ne{\v{s}}et{\v{r}}il, Jaroslav},
   author={R{\"o}dl, Vojt{\v{e}}ch},
   title={The Ramsey property for graphs with forbidden complete subgraphs},
   journal={J. Combinatorial Theory Ser. B},
   volume={20},
   date={1976},
   number={3},
   pages={243--249},
   review={\MR{0412004 (54 \#133)}},
}	

\bib{Sp57}{article}{
   author={Specker, Ernst},
   title={Teilmengen von Mengen mit Relationen},
   language={German},
   journal={Comment. Math. Helv.},
   volume={31},
   date={1957},
   pages={302--314},
   issn={0010-2571},
   review={\MR{0088454 (19,521b)}},
}
\end{biblist}
\end{bibdiv}
\end{document}